\documentclass[oneside,english]{amsart}
\usepackage[T1]{fontenc}
\usepackage[latin9]{inputenc}
\usepackage{color}
\usepackage{babel}
\usepackage{amsthm}
\usepackage{amssymb}
\usepackage{esint}
\usepackage[unicode=true,pdfusetitle,
 bookmarks=true,bookmarksnumbered=false,bookmarksopen=false,
 breaklinks=false,pdfborder={0 0 0},backref=false,colorlinks=true]
 {hyperref}

\makeatletter
\numberwithin{equation}{section}
\numberwithin{figure}{section}
\theoremstyle{plain}
\newtheorem{thm}{\protect\theoremname}
  \theoremstyle{definition}
  \newtheorem{defn}[thm]{\protect\definitionname}
  \theoremstyle{remark}
  \newtheorem*{rem*}{\protect\remarkname}
  \theoremstyle{definition}
  \newtheorem*{example*}{\protect\examplename}
  \theoremstyle{plain}
  \newtheorem{lem}[thm]{\protect\lemmaname}
  \theoremstyle{plain}
  \newtheorem{cor}[thm]{\protect\corollaryname}
  \theoremstyle{definition}
  \newtheorem{example}[thm]{\protect\examplename}
  \theoremstyle{plain}
  \newtheorem{prop}[thm]{\protect\propositionname}
  \theoremstyle{remark}
  \newtheorem*{claim*}{\protect\claimname}

\makeatother

  \providecommand{\claimname}{Claim}
  \providecommand{\corollaryname}{Corollary}
  \providecommand{\definitionname}{Definition}
  \providecommand{\examplename}{Example}
  \providecommand{\lemmaname}{Lemma}
  \providecommand{\propositionname}{Proposition}
  \providecommand{\remarkname}{Remark}
\providecommand{\theoremname}{Theorem}

\begin{document}

\title{Uniformly convex metric spaces}

\author{Martin Kell}
\begin{abstract}
In this paper the theory of uniformly convex metric spaces is developed.
These spaces exhibit a generalized convexity of the metric from a
fixed point. Using a (nearly) uniform convexity property a simple
proof of reflexivity is presented and a weak topology of such spaces
is analyzed. This topology called co-convex topology agrees with the
usualy weak topology in Banach spaces. An example of a $CAT(0)$-spaces
with weak topology which is not Hausdorff is given. This answers questions
raised by Monod 2006, Kirk and Panyanak 2008 and Esp\'inola and Fern\'andez-Le\'on
2009. 

In the end existence and uniqueness of generalized barycenters is
shown and a Banach-Saks property is proved.
\end{abstract}

\thanks{The author wants to thank Prof. Jürgen Jost and the MPI MiS for providing
a stimulating research environment. Also thanks to Miroslav Ba{\v{c}}\'ak
for helpful explanations of $CAT(0)$-space and remarks on an early
version of the paper which simplified some of the statements. }

\date{\today }

\email{mkell@mis.mpg.de}

\address{Max-Planck-Institute for Mathematics in the Sciences, Inselstr. 22,
04103 Leipzig, Germany}

\maketitle
\global\long\def\conv{\operatorname{conv}}
\global\long\def\esup{\operatorname*{ess\, sup}}
\global\long\def\supp{\operatorname{supp}}
\global\long\def\argmin{\operatorname*{arg\, min}}
\global\long\def\dom{\operatorname{dom}}
\global\long\def\sep{\operatorname{sep}}

\global\long\def\Var{\operatorname{Var}}
\global\long\def\diam{\operatorname{diam}}
\global\long\def\Lim{\operatorname{Lim}}

In this paper we summarize and extend some facts about convexities
of the metric from a fixed point and give simpler proofs which also
work for general metric spaces. In its simplest form this convexity
of the metric just requires balls to be convex or that $x\mapsto d(x,y)$
is convex for every fixed $y\in X$. It is easy to see that both conditions
are equivalent on normed spaces with strictly convex norm. However,
in \cite{Busemann1979} (see also \cite[Example 1]{Foertsch2004})
Busemann and Phadke constructed spaces whose balls are convex but
its metric is not. Nevertheless, a geometric condition called non-positive
curvature in the sense of Busemann (see \cite{Bridson1999,Bacak2014a})
implies that both concepts are equivalent, see \cite[Proposition 1]{Foertsch2004}.

The study of stronger convexities for Banach spaces \cite{Clarkson1936}
has a long tradition. In the non-linear setting so called $CAT(0)$-spaces
are by now well-understood, see \cite{Bridson1999,Bacak2014a}. Only
recently Kuwae \cite{Kuwae2013} based on \cite{Noar2011} studied
spaces with a uniformly $p$-convexity assumption similar to that
of Banach spaces. Related to this are Ohta's convexities definitions
\cite{Ohta2007} which, however, seem more restrictive than the ones
defined in this paper.

In the first section of this article an overview of convexities of
the metric and some easy implications are given. Then existence of
the projection map onto convex subsets and existence and uniqueness
of barycenters of measures is shown. For this we give simple proofs
using an old concept introduced by Huff in \cite{Huff1980}. 

In the third section we introduce weak topologies. The lack of a naturally
defined dual spaces similar to Banach space theory requires a more
direct definition either via convex sets, i.e. the co-convex topology
(first appeared in \cite{Monod2006}), or via asymptotic centers (see
historical remark at the end of \cite[Chapter 3]{Bacak2014a}). Both
topologies might not be equivalent and/or comparable. For $CAT(0)$-spaces
it is easy to show that the convergence via asymptotic centers is
stronger than the co-convex topology. However, the topologies do not
agree in general, see Example \ref{ex:cone-hilbert}. With this example
we answer questions raised in \cite{Kirk2008} and \cite{Espinola2009}. 

In the last sections, we use the results show existence of generalized
barycenters and prove the Banach-Saks property. The proof extends
a proof recently found by Yokota \cite{Yokota} in the setting of
$CAT(1)$-spaces with small diameter. In the end a discussion about
further extending convexities is given.

\section*{Convexity of the metric}

Let $(X,d)$ be a complete metric space. We say that $(X,d)$ admits
midpoints if for every $x,y\in X$ there is an $m(x,y)$ such that
$d(x,m(x,y))=d(y,m(x,y))=\frac{1}{2}d(x,y)$. One easily sees that
each such space is a geodesic space.

Now for $p\in[1,\infty)$ and all non-negative real numbers $a,b$
we define the $p$-mean 
\[
\mathcal{M}^{p}(a,b):=\left(\frac{1}{2}a^{p}+\frac{1}{2}b^{p}\right)^{\frac{1}{p}}.
\]
 Furthermore, the case $p=\infty$ can be defined as a limit, i.e.
$\mathcal{M}_{t}^{\infty}(a,b)=\max\{a,b\}$.
\begin{defn}
[$p$-convexity]Suppose the metric space $(X,d)$ admits midpoints.
Then it is called $p$-convex for some $p\in[1,\infty]$ if for each
triple $x,y,z\in X$ and each midpoint $m(x,y)$ of $x$ and $y$
\[
d(m(x,y),z)\le\mathcal{M}^{p}(d(x,z),d(y,z)).
\]

It is called strictly $p$-convex for $p\in(1,\infty]$ if the inequality
is strict whenever $x\ne y$ and strictly $1$-convex if the inequality
is strict whenever $d(x,y)>|d(x,z)-d(y,z)|$.\end{defn}
\begin{rem*}
(1) In \cite{Foertsch2004} Foertsch defined $1$-convexity and $\infty$-convexity
under the name distance convexity and ball convexity.
\end{rem*}
A $p$-convex space is $p'$-convex for all $p'\ge p$ and it is easy
to see that balls are convex iff the space is $\infty$-convex. Furthermore,
one sees that any strictly $\infty$-convex space is uniquely geodesic.

Instead of just requiring convexity from a fixed point one can assume
that assume a convexity of $t\mapsto d(x_{t},y_{t})$ there $x_{t}$
and $y_{t}$ are constant speed geodesics. This gives the following
condition.
\begin{defn}
[$p$-Busemann curvature] A metric space $(X,d)$ admitting midpoints
is said to satisfy the $p$-Busemann curvature condition for some
$p\in[1,\infty]$ if for all triples quadruples $x_{0},x_{1},y_{0},y_{1}\in X$
with midpoints $x_{\frac{1}{2}}=m(x_{0},x_{1})$ and $y_{\frac{1}{2}}=m(y_{0},y_{1})$
it holds 
\[
d(x_{\frac{1}{2}},y_{\frac{1}{2}})\le\mathcal{M}^{p}(d(x_{0},y_{0}),d(x_{1},y_{1})).
\]
In such a case we will say that $(X,d)$ is $p$-Busemann.
\end{defn}
It is not difficult to show (see e.g. \cite[Proposition 1.1.5]{Bacak2014a})
that in case $p\in[1,\infty)$ this is equivalent to the more traditional
form: for each triples $x,y,z\in X$ with midpoints $m_{1}=m(x,z)$
and $m_{2}=m(y,z)$ it holds 
\[
d(m_{1},m_{2})^{p}\le\frac{1}{2}d(x,y)^{p}.
\]
In particular for $p=1$, this is Busemann's original non-positive
curvature assumption. In this case we will just say that $(X,d)$
is Busemann. Busemann's condition can be used to show equivalence
of all (strict/uniform) $p$-convexity, see Corollary \ref{cor:Busemann-conv}.
Currently we cannot prove that on $p$-Busemann spaces (strict/uniform)
$p$-convexity is equivalent to (strict/uniform) $p'$-convexity for
all $p'\ge p$. However, it can be used to get a $p$-Wasserstein
contraction of $2$-barycenters if Jensen's inequality holds on the
space, see Proposition \ref{prop:Jensen-p-Busemann} below.

In \cite{Foertsch2004} Foertsch also defines uniform distance/ball
convexity. We adapt his definition as follows:
\begin{defn}
[uniform $p$-convexity] Suppose $(X,d)$ admits midpoints and let
$p\in[1,\infty]$. Then we say it is uniformly $p$-convex if for
all $\epsilon>0$ there is a $\rho_{p}(\epsilon)\in(0,1)$ such that
for all triples $x,y,z\in X$ satisfying $d(x,y)>\epsilon\mathcal{M}^{p}(d(x,z),d(y,z))$
for $p>1$ and $d(x,y)>|d(x,z)-d(y,z)|+\epsilon\mathcal{M}^{1}(d(x,z),d(y,z))$
for $p=1$ it holds
\[
d(m(x,y),z)\le(1-\rho_{p}(\epsilon))\mathcal{M}^{p}(d(x,z),d(y,z)).
\]
\end{defn}
\begin{rem*}
(1) W.l.o.g. we assume that $\rho_{p}$ is monotone in $\epsilon$
so that $\rho_{p}(\epsilon)\to0$ requires $\epsilon\to0$.

(2) Uniform $p$-convexity for $p\in(1,\infty)$ is equivalent to
the existence of a $\tilde{\rho}_{p}(\epsilon)>0$ such that 
\[
d(m(x,y),z)^{p}\le(1-\tilde{\rho}_{p}(\epsilon))\mathcal{M}^{p}\left(d(x,z),d(x,z)\right)^{p},
\]
just let $\tilde{\rho}_{p}(\epsilon)=1-(1-\rho_{p}(\epsilon))^{p}$.

(3) The usual definition of uniform convexity for functions is as
follows: A function $f$ is uniformly convex if for $x,y\in X$ with
midpoint $m$: 
\[
f(m)\le\frac{1}{2}f(x)+\frac{1}{2}f(y)-\omega(d(x,y)),
\]
where $\omega$ is the modulus of convexity and $\omega(r)>0$ if
$r>0$. For $p\ge2$ and $\omega(r)=Cr^{p}$ and $f=d(\cdot,z)^{p}$
one recovers Kuwae's $p$-uniform convexity \cite{Kuwae2013}. However,
in this form one does not see whether $p$-convexity implies $p'$-convexity.
Furthermore, one gets a restriction that $\omega(r)\ge Cr^{2}$, i.e.
the cases $p\in(1,2)$ are essentially excluded. And whereas our definition
is multiplicative, matching the fact that $Cd(\cdot,\cdot)$ is also
a metric, the usual uniform convexity is only multiplicative by adjusting
the modulus of convexity. \end{rem*}
\begin{example*}
(1) Every $CAT(0)$-space is uniformly $2$-convex with $\tilde{\rho}(\epsilon)=\left(\frac{\epsilon}{2}\right)^{2}$.
More generally any $R_{\kappa}$-domain of a $CAT(\kappa)$-space
is uniformly $2$-convex with $\rho(\epsilon)=c_{\kappa}\epsilon^{2}$.

(2) Every $p$-uniformly convex space as defined in \cite{Noar2011,Kuwae2013}
is uniformly $p$-convex with $\rho(\epsilon)=c_{k}\epsilon{}^{p}$.\end{example*}
\begin{lem}
A uniformly $p$-convex metric space $(X,d)$ is uniformly $p'$-convex
for all $p'\ge p$.\end{lem}
\begin{proof}
First note that $\mathcal{M}^{p}(a,b)\le\mathcal{M}^{p'}(a,b)$. 

Assume first that $1<p<p'\le\infty$. If $x,y,z\in X$ is a triple
satisfying the condition for $p'$ then it also satisfies the condition
for $p$ and thus for $m=m(x,y)$ 
\begin{eqnarray*}
d(m,z) & \le & (1-\rho_{p}(\epsilon))\mathcal{M}^{p}(d(x,z),d(y,z))\\
 & \le & (1-\rho_{p}(\epsilon))\mathcal{M}^{p'}(d(x,z),d(y,z)).
\end{eqnarray*}
Hence setting $\rho_{p'}(\epsilon):=\rho_{p}(\epsilon)$ gives the
result.

For $p=1$ we skip the case $p'=\infty$ as this was proven in \cite[Proposition 1]{Foertsch2004}:
Let $x,y,z\in X$ be some triple with $d(x,z)\ge d(y,z)$. If 
\[
d(x,y)>|d(x,z)-d(y,z)|+\frac{\epsilon}{2}\mathcal{M}^{1}(d(x,z),d(y,z))
\]
then 
\begin{eqnarray*}
d(m(x,y),z) & \le & \left(1-\rho_{1}\left(\frac{\epsilon}{2}\right)\right)\mathcal{M}^{1}(d(x,z),d(y,z))\\
 & \le & \left(1-\rho_{1}\left(\frac{\epsilon}{2}\right)\right)\mathcal{M}^{p'}(d(x,z),d(y,z)).
\end{eqnarray*}
So assume 
\[
d(x,y)\le d(x,z)-d(y,z)+\frac{\epsilon}{2}\mathcal{M}^{p'}(d(x,z),d(y,z)).
\]
If $d(x,z)-d(y,z)\le\frac{\epsilon}{2}\mathcal{M}^{p'}(d(x,z),d(y,z))$
then 
\[
d(x,y)\le\epsilon\mathcal{M}^{p'}(d(x,z),d(y,z)).
\]
Hence we can assume $d(x,z)-d(y,z)>\frac{\epsilon}{2}\mathcal{M}^{p'}(d(x,z),d(y,z))$.
Now for $p'\ge2$ Clarkson's inequality
\[
\left(\frac{1}{2}a+\frac{1}{2}b\right)^{p'}+c_{p'}(a-b)^{p'}\le\frac{1}{2}a^{p'}+\frac{1}{2}b^{p'}
\]
holds. Thus using $1$-convexity and our assumption we get
\begin{eqnarray*}
d(m,z)^{p'} & \le & \left(\frac{1}{2}d(x,z)+\frac{1}{2}d(y,z)\right)^{p'}\\
 & \le & \frac{1}{2}d(x,z)^{p'}+\frac{1}{2}d(y,z)^{p'}-c_{p'}(d(x,z)-d(y,z))^{p'}\\
 & \le & \left(1-c_{p}\left(\frac{\epsilon}{2}\right)^{p'}\right)\mathcal{M}^{p'}(d(x,z),d(y,z))^{p'}.
\end{eqnarray*}
Choosing $\rho_{p'}(\epsilon)=\min\{\rho_{1}(\frac{\epsilon}{2}),1-(1-c_{p}(\frac{\epsilon}{2})^{p'})^{\frac{1}{p'}}\}$
gives the result.

For $1<p'<2$ we use the other Clarkson inequality
\[
\left(\frac{1}{2}a+\frac{1}{2}b\right)^{q}+c_{p'}(a-b)^{q}\le\left(\frac{1}{2}a^{p}+\frac{1}{2}b^{p}\right)^{q}
\]
where $\frac{1}{q}+\frac{1}{p'}=1$. By similar arguments we get 
\[
d(m,z)^{q}+c_{p'}\left(\frac{\epsilon}{2}\right)^{\frac{q}{p'}}(\mathcal{M}^{p'}(d(x,z),d(y,z))^{q}\le(\mathcal{M}^{p'}(d(x,z),d(y,z))^{q}.
\]
Choosing in this case 
\[
\rho_{p'}(\epsilon)=\min\left\{ \rho_{1}\left(\frac{\epsilon}{2}\right),1-\left(1-c_{p'}\left(\frac{\epsilon}{2}\right)^{\frac{q}{p'}}\right)^{\frac{1}{q}}\right\} 
\]
finishes the proof.\end{proof}
\begin{cor}
\label{cor:Busemann-conv}Assume $(X,d)$ is Busemann. Then $(X,d)$
is (strictly/uniformly) $p$-convex for some $p\in[1,\infty]$ iff
it is (strictly/uniformly) $p$-convex for all $p\in[1,\infty]$.\end{cor}
\begin{proof}
This is just a using \cite[Proposition 1]{Foertsch2004} who proved
that (strict/uniform) $\infty$-convexity implies (strict/uniform)
$1$-convexity. 
\end{proof}
Any $CAT(0)$-space is both Busemann and uniformly $2$-convex, hence
uniformly $p$-convex for every $p\in[1,\infty]$.

\section*{Convex subsets and reflexivity}

In a geodesic metric space, we say that subset $C\subset X$ is convex
if for each $x,y\in C$ and each geodesic $\gamma$ connecting $x$
and $y$ also $\gamma\subset C$. Given any subset $A\subset X$ we
define the convex hull of $A$ as follows: $G_{0}=A$ then for $n\ge1$
\[
G_{n}=\bigcup_{x,y\in G_{n-1}}\{\gamma_{t}\,|\,\gamma\mbox{ is a geodesic connecting \ensuremath{x\:}and \ensuremath{y\:}and \ensuremath{t\in[0,1]\}}}
\]
\[
\conv A=\bigcup_{n\in\mathbb{N}}G_{n}.
\]
 The closed convex hull is just the closure $\overline{\conv A}$
of $\conv A$.

The projection map onto (convex) sets can be defined as follows: Given
a non-empty subset $C$ of $X$ define $r_{C}:X\to[0,\infty)$ by
\[
r_{C}(x)=\inf_{c\in C}d(x,c)
\]
and $P_{C}:X\to2^{C}$ by 
\[
P_{C}(x)=\{c\in C\,|\, r_{C}(x)=d(x,c)\}.
\]
In case $|P_{C}(x)|=1$ for all $x\in X$ we say that the set $C$
is Chebyshev. In that case, just assume $P_{C}$ is a map from $X$
to $C$. 

It is well-known that a Banach space is reflexive iff any decreasing
family of closed bounded convex subsets has non-empty intersection.
Thus it makes sense for general metric spaces to define reflexivity
as follows. 
\begin{defn}
[Reflexivity] A metric space $(X,d)$ is said to be reflexive if
for every decreasing family $(C_{i})_{i\in I}$ of non-empty bounded
closed convex subsets, i.e. $C_{i}\subset C_{j}$ whenever $i>j$
where $I$ is a directed set then it holds 
\[
\bigcap_{i\in I}C_{i}\ne\varnothing.
\]

\end{defn}
It is obvious that any proper metric space is reflexive. The following
was defined in \cite{Huff1980}. We will simplify Huff's proof of
\cite[Theorem 1]{Huff1980} to show that nearly uniform convexity
implies reflexivity using a proof via the projection map, see e.g.
\cite[Proofs of 2.1.12(i) and 2.1.16]{Bacak2014a}. However, since
the weak topology (see below) is not necessarily Hausdorff, we cannot
show that nearly uniform convexity also implies the uniform Kadec-Klee
property.

We say that a family of points $(x_{i})_{i\in I}$ is $\epsilon$-separated
if $d(x_{i},x_{j})\ge\epsilon$ for $i\ne j$, i.e. 
\[
\sep((x_{i})_{i\in I})=\inf d(x_{i},x_{j})\ge\epsilon.
\]

\begin{defn}
[Nearly uniformly convex] A $\infty$-convex metric space $(X,d)$
is said to be nearly uniformly convex, if for any $R>0$ for any $\epsilon$-separated
infinite family $(x_{i})_{i\in I}$ with $d(x_{i},y)\le r\le R$ there
is a $\rho=\rho(\epsilon,R)>0$ such that 
\[
B_{(1-\rho)r}(y)\cap\overline{\conv(x_{i})_{i\in I}}\ne\varnothing.
\]

\end{defn}
Note that uniform $\infty$-convexity implies nearly uniform convexity,
an even stronger statement is formulated in Theorem \ref{thm:NUC}.
However, not every nearly uniformly convex space is uniformly convex,
see \cite{Huff1980}.
\begin{thm}
For every closed convex subset $C$ of a nearly uniformly convex metric
space the projection $P_{C}$ has non-empty compact images, i.e. $P_{c}(x)$
is non-empty and compact for every $x\in X$.\end{thm}
\begin{cor}
If $(X,d)$ is nearly uniformly convex and strictly $\infty$-convex
then every closed convex set is Chebyshev.\end{cor}
\begin{proof}
[Proof of the Theorem]Let $C$ be a closed convex subset, $x\in X$
be arbitrary and set $r=r_{C}(x)$. For each $n\in N$ there is an
$x_{n}\in C$ such that $r\le d(x,x_{n})\le r+\frac{1}{n}$. In particular,
$d(x,x_{n})\to r$ as $n\to\infty$ . If $r=0$ or every subsequence
of $(x_{n})$ admits a convergent subsequence we are done. 

So assume $(x_{n})$ w.l.o.g. that $(x_{n})$ is $\epsilon$-separated
for some $\epsilon>0$. By nearly uniform convexity there is a $\rho=\rho(\epsilon)>0$
such that
\[
A_{n}=B_{(1-\rho)(r+\frac{1}{n})}(x)\cap\overline{\conv(x_{m})_{m\ge n}}\ne\varnothing.
\]
For sufficiently large $n$ and some $0<\rho'<\rho$ we also have
$B_{(1-\rho)(r+\frac{1}{n})}(x)\subset B_{(1-\rho')r}(x)$, i.e. $d(x,y)<r$
for some $y\in A_{n}$. But this contradicts the fact that $\overline{\conv(x_{m})_{m\ge n}}\subset C$,
i.e. $d(x,y)\ge r$ for all $y\in A_{n}$. \end{proof}
\begin{thm}
A nearly uniformly convex metric space is reflexive.\end{thm}
\begin{proof}
Let $(C_{i})_{i\in I}$ be a non-increasing family of bounded closed
convex subsets of $X$ and let $x\in X$ be some arbitrary point.
For each $i\in I$ define $r_{i}=\inf_{y\in C_{i}}d(x,y)$. Since
$(C_{i})_{i\in I}$ is non-increasing so the net $(r_{i})_{i\in I}$
is non-decreasing and bounded, hence convergent to some $r$. By the
previous theorem there are $x_{i}\in C_{i}$ such that $d(x,x_{i})=r_{i}$.
If $r=0$ or $(x_{i})_{i\in I}$ admits a convergent subnet we are
done.

So assume there is an $\epsilon$-separated subnet $(x_{i'})_{i'\in I'}$
for some $\epsilon>0$. Now nearly uniform convexity implies that
for some $\rho=\rho(\epsilon)>0$ 
\[
\varnothing\ne A_{i}=B_{(1-\rho)r}(x)\cap\overline{\conv(x_{j})_{j\ge i}}\subset C_{i}.
\]
Since the subnet $(r_{i'})$ is also convergent to $r$ there is some
$i$ and $0<\rho'<\rho$ such that
\[
B_{(1-\rho)r}(x)\subset B_{(1-\tilde{\rho})r_{i}}(x).
\]

However, this implies that $d(x,y_{i})<r_{i}$ for all $y_{i}\in A_{i}$
contradicting the definition of $r_{i}$.
\end{proof}
In order to use reflexivity to characterize the weak topology defined
below better we need the following equivalent description. We say
that a collection of sets $(C_{i})_{i\in I}$ has the finite intersection
property if any finite subcollection has non-empty intersection, i.e.
for every finite $I'\subset I$, $\cap_{i\in I'}C_{i}\ne\varnothing$.
\begin{lem}
\label{cor:reflex-equiv}The space $(X,d)$ is reflexive iff every
collection $(C_{i})_{i\in I}$ of closed bounded convex subsets with
finite intersection property satisfies 
\[
\bigcap_{i\in I}C_{i}\ne\varnothing.
\]
\end{lem}
\begin{proof}
The if-direction is obvious. So assume $(X,d)$ is reflexive and $(C_{i})_{i\in I}$
be a collection of closed bounded convex subsets with finite intersection
property.

Let $\mathcal{I}$ be the set of finite subsets of $I$. This set
directed by inclusion and the sets 
\[
\tilde{C}_{\mathbf{i}}=\bigcap_{i\in\mathbf{i}}C_{i}
\]
are non-empty closed and convex. Furthermore, the family $(\tilde{C}_{\mathbf{i}})_{\mathbf{i}\in\mathcal{I}}$
is decreasing. By reflexivity
\[
\bigcap_{i\in I}C_{i}=\bigcap_{\mathbf{i}\in\mathcal{I}}\tilde{C}_{\mathbf{i}}\ne\varnothing.
\]
 \end{proof}
\begin{thm}
Assume $(X,d)$ is strictly $\infty$-convex and nearly uniformly
convex. Then the midpoint map $m$ is continuous.\end{thm}
\begin{proof}
Since $(X,d)$ is strictly $\infty$-convex we see that geodesics
are unique. Thus the midpoint map by $m:(x,y)\mapsto m(x,y)$ is well-defined.
Now if $(x_{n},y_{n})\to(x,y)$ then for all $\epsilon>0$ the sequence
$m_{n}=m(x_{n},y_{n})$ eventually enters the closed convex and bounded
set 
\[
A_{\epsilon}=B_{\frac{1}{2}d(x,y)+\epsilon}(x)\cap B_{\frac{1}{2}d(x,y)+\epsilon}(y).
\]
By uniform $\infty$-convexity $\bigcap_{\epsilon>0}A_{\epsilon}$
is non-empty and contains only the point $m(x,y)$. 

We only need to show that $\diam A_{\epsilon}\to0$ as $\epsilon\to0$.
Now assume there is a sequence $x_{n}\in A_{\frac{1}{n}}$ that is
not Cauchy, so assume it is $\delta$-separated for some $\delta>0$.
Then by nearly uniform convexity there is a $\rho(\delta)>0$ 
\[
B_{(1-\rho)\left(\frac{1}{2}d(x,y)+\epsilon\right)}(x)\cap\overline{\conv(x_{n})}\ne\varnothing.
\]
And thus 
\[
\bigcap_{m}\overline{\conv(x_{n})_{n\ge m}}\cap B_{(1-\rho(\delta))\frac{1}{2}d(x,y)}(x)\ne\varnothing.
\]
But this contradicts the fact that 
\[
\bigcap_{m}\overline{\conv(x_{n})_{n\ge m}}\subset\bigcap_{m}A_{\frac{1}{m}}\subset B_{\frac{1}{2}d(x,y)}(y)
\]
is disjoint from $B_{(1-\rho)\left(\frac{1}{2}d(x,y)+\epsilon\right)}(x)$.\end{proof}
\begin{cor}
A strictly $\infty$-convex, nearly uniformly convex metric space
is contractible.\end{cor}
\begin{proof}
Take a fixed point $x_{0}\in X$ and define the map 
\[
\Phi_{t}(x)=\gamma_{xx_{0}}(t)
\]
where $\gamma_{xx_{0}}$ is the geodesic connecting $x$ and $x_{0}$.
Now proof of previous theorem also shows that $t$-midpoints are continuous,
in particular $\Phi_{t}$ is continuous.
\end{proof}

\section*{Weak topologies}

In Hilbert and Banach spaces the concept of weak topologies can be
introduced with the help of dual spaces. Since for general metric
spaces there is (by now) no concept of dual spaces, a direct definition
needs to be introduced. As it turns out the first topology agrees
with the usual weak topology, see Corollary \ref{cor:Banach-weak-co-convex}.

\subsection*{Co-convex topology}

The first weak topology on metric spaces is the following. It already
appeared in \cite{Monod2006}. As it turns out, this topology is agrees
with the weak topology on any Banach space, see Corollary \ref{cor:Banach-weak-co-convex}
below. 
\begin{defn}
[Co-convex topology] Let $(X,d)$ be a metric space. Then the co-convex
topology $\tau_{co}$ is the weakest topology containing all complements
of closed convex sets. 
\end{defn}
Obviously this topology is weaker than the topology induced by the
metric and since point sets are convex the topology satisfies the
$T_{1}$-separation axiom, i.e. for each two points $x,y\in X$ there
is an open neighborhood $U_{x}$ containing $x$ but not $y$. Furthermore,
the set of weak limit points of a sequence $(x_{n})_{n\in\mathbb{N}}$
is convex if the space is $\infty$-convex. A useful characterization
of the limit points is the following: 
\begin{lem}
\label{lem:convex-leaving}A sequence of points $x_{n}$ converges
weakly to $x$ iff for all subsequences $(x_{n'})$ it holds 
\[
x\in\overline{\conv(x_{n'})}.
\]
The set of limit point $\Lim(x_{n})$ is the non-empty subset
\[
\bigcap_{(i_{n})\subset I_{\inf}}\overline{\conv(x_{i_{n}})}
\]
where $I_{\inf}$ is the set of sequences of increasing natural numbers.\end{lem}
\begin{rem*}
The same statement holds for also for nets. Below we will make most
statments only for sequences if in fact they also hold for nets.\end{rem*}
\begin{proof}
This follows immediately from the fact that 
\[
A(x_{n'})=\overline{\conv(x_{n'})}
\]
is closed, bounded and convex and thus weakly closed. 

First suppose $x\notin A(x_{n'})$ for some subsequence $(x_{n'})$.
By definition $x_{n}\overset{\tau_{co}}{\to}y$ implies that $x_{n}$
eventually leaves every closed bounded convex sets not containing
$y$. Since $(x_{m'})\subset A(x_{n'})$ for $m'\ge n'$, we conclude
$(x_{n'})$ cannot converge weakly to $x$.

Conversely, if $(x_{n})$ does not converge to $x$ then there is
a weakly open set $U\in\tau_{co}$ such that $(x_{n})\not\subset U$
and $x\in U$. In particular, for some subsequence $(x_{n'})$ it
holds $(x_{n'})\subset X\backslash U$. Since $\tau_{co}$ is generated
by complements of closed convex sets we can assume $U=X\backslash C$
for some closed convex subset $C$. Therefore, $(x_{n'})\subset C$
and thus $A(x_{n'})\subset C$, i.e. $x\notin A(x_{n'})$.\end{proof}
\begin{cor}
For any weakly convergent sequence $(x_{n})$ and countable subset
$A$ disjoint from $\Lim(x_{n})$ there is a subsequence $(x_{n'})$
such that 
\[
A\cap\bigcap_{m\in\mathbb{N}}\overline{\conv(x_{n'})_{n'\ge m}}=\varnothing.
\]
\end{cor}
\begin{proof}
First note, by the lemma above there is a subsequence $(x_{m_{n}^{(0)}})$
of $(x_{n})$ such that $y_{0}\in A$ is not contained in $\overline{\conv(x_{m_{n}^{(0)}})}$.
Now inductively constructing $(x_{m_{n}^{(k)}})$ avoiding $y_{k}$
using the sequence $(x_{m_{n}^{(k-1)}})$ we can choose the diagonal
sequence $m_{n}=m_{n}^{(n)}$ such that 
\[
y\notin\bigcap_{m\in\mathbb{N}}\overline{\conv(x_{m_{n}})_{m_{n}\ge m}}
\]
for all $y\in A$.\end{proof}
\begin{cor}
\label{cor:Banach-weak-co-convex}On any Banach space $X$ the co-convex
topology $\tau_{co}$ agrees the weak topology $\tau_{w}$. In particular,
$\tau_{co}$ is Hausdorff.\end{cor}
\begin{proof}
By Corollary \ref{cor:quasi-conv+conc} below any linear functional
$\ell\in X^{*}$ is $\tau_{co}$-continuous. Hence $x_{n}\overset{\tau_{co}}{\to}x$
implies $x_{n}\overset{\tau_{w}}{\to}x$. The converse follows from
that fact that for any subsequence$(x_{n'})$ the set $\overline{\conv(x_{n'})}$
is $\tau_{w}$-closed and $x_{n'}\overset{\tau_{w}}{\to}x$. Therefore,
$x\in\overline{\conv(x_{n'})}$ which implies $x_{n}\overset{\tau_{co}}{\to}x$
by Lemma \ref{lem:convex-leaving} above.
\end{proof}
Now similar to Banach spaces, one can easy show that reflexivity implies
weak compactness of bounded closed convex subsets.
\begin{thm}
Bounded closed convex subsets are weakly compact iff the space is
reflexive.\end{thm}
\begin{proof}
By Alexander sub-base theorem it suffices to show that each open cover
$(U_{i})_{i\in I}$ of $B$, where $U_{i}$ is a complement of a closed
convex set, has a finite subcover. For this, note that $U_{i}=X\backslash C_{i}$
and the cover property of $U_{i}$ is equivalent to 
\[
\bigcap_{i\in I}B\cap C_{i}=\varnothing.
\]
If we assume that there is no finite subcover then the collection
$(B\cap C_{i})_{i\in I}$ has finite intersection property. But then
Corollary \ref{cor:reflex-equiv} yields $\bigcap_{i\in I}B\cap C_{i}\ne\varnothing$,
which is a contradiction. 

Conversely, assume $(X,d)$ is not reflexive but any bounded closed
convex subset is weakly compact. Then there $(C_{i})_{i\in I}$ is
a decreasing family of non-empty bounded closed convex subsets such
that $\cap_{i\in I}C_{i}=\varnothing$. Assume w.l.o.g. that $I$
has a minimal element $i_{0}$. Then $U_{i}=X\backslash C_{i}$ is
an open cover of $C_{i_{0}}$, i.e. 
\[
C_{i_{0}}\subset\bigcup_{i\in I}U_{i}.
\]
Since $(C_{i})_{i\in I}$ is decreasing, $(U_{i})_{i\in I}$ is increasing.
By weak compactness, finitely many of there are sufficient to cover
$C_{i_{0}}$. Since $(U_{i})_{i\in I}$ is increasing, there exists
exactly one $i_{1}\in I$ such that $C_{i_{0}}\subset U_{i_{1}}=X\backslash C_{i_{1}}$.
But then $C_{i_{1}}=\varnothing$ contradicting our assumption.
\end{proof}
Note that on general spaces the co-convex topology is not necessary
Hausdorff. Even in case of $CAT(0)$-spaces one can construct an easy
counterexample.
\begin{example}
[Euclidean Cone of a Hilbert space]\label{ex:cone-hilbert} For the
construction of Euclidean cones see \cite[Chapter I.5]{Bridson1999}.
Let $(H,d_{H})$ be an infinite-dimensional Hilbert space and $d_{H}$
be the induced metric. The Euclidean cone over $(H,d_{h})$ is defined
as the set $C(H)=H\times[0,\infty)$ with the metric 
\[
d((x,t),(x',t'))^{2}:=t^{2}+t'^{2}-2tt'\cos(d_{\pi}(x,x'))
\]
where $d_{\pi}(x,x')=\min\{\pi,d_{H}(x,x')\}$. By \cite[Theorem II-3.14]{Bridson1999}
$(C(H),d)$ is a $CAT(0)$-space and thus uniformly $p$-convex for
any $p\in[1,\infty]$. In particular, bounded closed convex subsets
are compact w.r.t. the co-convex topology. Note that in $(H,d_{H})$
the co-convex topology agrees with the usual weak topology. Now let
$((e_{n},1))_{n\in\mathbb{N}}$ be a sequence in $C(H)$. We claim
that for any subsequence $((e_{n'},1))$ we have 
\[
\bigcap_{m\in\mathbb{N}}\overline{\conv((e_{n'},1))_{n'\ge m}}=\{(\mathbf{0},r)\,|\, r\in[a,b]\}
\]
with $a<b$ where it is easy to see that $a$ and $b$ do not depend
on the subsequence. Any point in that intersection is a limit point
of $((e_{n},1))$ which implies that $\tau_{co}(C(H))$ is not Hausdorff.
To see this, note that the projection $p$ onto the line $\{(\mathbf{0},r)\,|\, r\ge0\}$
has the following form
\[
p((x,r))=(\mathbf{0},r\cos(d_{H}(x,\mathbf{0})))
\]
for $d(x,0)\le\frac{\pi}{2}$. In particular, $d((e_{n},1))=(\mathbf{0},\cos(1))$.
Using the weak sequential convergence defined below, this means that
$(e_{n},1)\overset{w}{\to}(\mathbf{0},\cos(1))$, in particular.

\[
(\mathbf{0},\cos(1))\in\bigcap_{m\in\mathbb{N}}\overline{\conv((e_{n'},1))_{n'\ge m}}.
\]
Now we will show that the sequence of midpoints $l_{mn}$ of $(e_{n},1)$
and $(e_{m},1)$ with $m\ne n$ converges weakly sequentially to some
point $(\mathbf{0},r)$ with $r>\cos(1)$. This immediately implies
that $\bigcap_{m\in\mathbb{N}}\overline{\conv((e_{n'},1))_{n'\ge m}}$
contains more than one point and each is a limit point of $(e_{n},1)$
w.r.t. the co-convex topology.

To show that $l_{mn}$ does not weakly sequentially converge to $(\mathbf{0},\cos(1))$
we just need to show that $p(l_{mn})\ne(\mathbf{0},\cos(1))$. By
the calculus of Euclidian cones the points $l_{mn}$ have the following
form 
\[
l_{mn}=\left(\frac{e_{m}-e_{n}}{2},r_{\frac{1}{2}}\right)
\]
where $r_{\frac{1}{2}}$ is the (positive) solution of the equation
\[
r^{2}+1-2r\cos\left(\frac{\sqrt{2}}{2}\right)=\frac{1}{4}\left(2-2\cos(\sqrt{2})\right),
\]
i.e. $r_{\frac{1}{2}}=\cos\left(\frac{\sqrt{2}}{2}\right)$.

Then the projection has the form 
\begin{eqnarray*}
p(l_{mn}) & = & \left(\mathbf{0},r_{\frac{1}{2}}\cos\left(\left\Vert \frac{e_{n}-e_{m}}{2}\right\Vert \right)\right)\\
 & = & \left(\mathbf{0},\cos\left(\frac{\sqrt{2}}{2}\right)^{2}\right).
\end{eqnarray*}
Since $\cos(1)<\cos(\frac{\sqrt{2}}{2})^{2}$ we see that $l_{mn}\not\to(\mathbf{0},\cos(1))$
w.r.t. weak sequential convergence and 
\[
(\mathbf{0},\cos(\frac{\sqrt{2}}{2})^{2})\in\bigcap_{m\in\mathbb{N}}\overline{\conv((e_{n'},1))}.
\]
And this obviously does not depend on the subsequence.

Note that this space also violates the property $(N)$ defined in
\cite{Espinola2009}, more generally any cone over a (even proper)
$CAT(1)$-space which is not the sphere gives a counterexample. The
example also gives a negative answer to Question 3 of \cite{Kirk2008}.
This topology is also a counterexample to topologies similar to Monod's
$\mathcal{T}_{w}$ topology: Let $\tau_{w}^{p}$ be the weakest topology
making all maps $x\mapsto d(x,y)^{p}-d(x,z)^{p}$ for $y,z\in X$
continuous. For Hilbert spaces and $p=2$ this is the weak topology,
(compare to \cite[18. Example]{Monod2006} which should be $p=2$).
For the space $(C(H),d))$ one can show that each $\tau_{w}^{p}$
is strictly stronger that the weak sequential convergence.\end{example}
\begin{defn}
[weak lower semicontinuity] A function $f:X\to(-\infty,\infty]$
is said to be weakly l.s.c. at a given point $x\in\dom f$ if 
\[
\liminf f(x_{i})\ge f(x)
\]
whenever $(x_{i})$ is a net converging to $x$ w.r.t. $\tau_{co}$.
We say $f$ is weakly l.s.c. if it is weakly l.s.c. at every $x\in\dom f$.\end{defn}
\begin{rem*}
A priori it is not clear if $\tau_{co}$ is first-countable and thus
the continuity needs to be stated in terms of nets. In that case it
boils down to $\liminf_{n\to\infty}f(x_{n})\ge f(x)$.\end{rem*}
\begin{prop}
\label{prop:co-convex-lscts}Assume $(X,d)$ is $\infty$-convex.
Then every lower semicontinuous quasi-convex function is weakly lower
semicontinuous. In particular, the metric is lower semicontinuous.\end{prop}
\begin{rem*}
A function is quasi-convex iff its sublevels are convex, i.e. whenever
$z$ is on a geodesic connecting $x$ and $y$ then $f(z)\le\max\{f(x),f(y)\}$. \end{rem*}
\begin{proof}
By definition of the co-convex topology, if $x_{i}\overset{\tau_{co}}{\to}x$
and $x_{i}\in C$ for some closed convex subset $C$ then $x\in C$.
Now assume $f$ is not weakly lower semicontinuous at $x$, i.e. 
\[
\liminf f(x_{i})<f(x).
\]
Then there is a $\delta>0$ such that 
\[
x_{i}\in A_{\delta}=\{y\in X\,|\, f(y)\le f(x)-\delta\}
\]
for all $i\ge i_{0}$. By quasi-convexity and lower semicontinuity
the set $A_{\delta}$ is closed convex and thus $x\in A_{\delta}$
which is a contradiction. Hence $f$ is weakly lower semicontinuous.
\end{proof}
A function $\ell:X\to\mathbb{R}$ is called quasi-monotone iff it
is both quasi-convex and quasi-concave. Similarly $\ell$ is called
linear iff it is both convex and concave. A linear function is obviously
quasi-monotone. The converse is not true in general: Every $CAT(0)$-spaces
with property $(N)$ (see \cite{Espinola2009}) admits such functionals;
for $x,y\in X$ just set $\ell(x')=d(P_{[x,y]}x',x)$ where $P_{[x,y]}$
is the projection onto the geodesic connecting $x$ and $y$. 
\begin{cor}
\label{cor:quasi-conv+conc}Assume $(X,d)$ is $\infty$-convex. Then
every continuous quasi-monotone function is weakly continuous. \end{cor}
\begin{proof}
Just note that the previous theorem implies that a quasi-monotone
function is both weakly lower and upper semicontinuous.
\end{proof}
In order to get the Kadec-Klee property one needs to find limit points
which are easily representable. 
\begin{defn}
[countable reflexive] A reflexive metric space $(X,d)$ is called
countable reflexive if for each weakly convergent sequence $(x_{n})$
there is a subsequence $(x_{n'})$ such that 
\[
\Lim(x_{n})=\bigcap_{m\in\mathbb{N}}\overline{\conv(x_{n'})_{n'\ge m}}.
\]

\end{defn}
By diagonal procedure it is easy to see that one only needs to show
that for each $\epsilon>0$ there is a subsequence $(x_{n'})$ such
that
\[
B_{\epsilon}(\Lim(x_{n}))\supset\bigcap_{m\in\mathbb{N}}\overline{\conv(x_{n'})_{n'\ge m}}.
\]

\begin{lem}
\label{lem:coun-refl-quasi-mono}Any reflexive Banach space is countable
reflexive. More generally any reflexive metric space admitting quasi-monotone
functions separating points is countable reflexive. In this case the
co-convex topology is Hausdorff. \end{lem}
\begin{proof}
If $x_{n}\overset{\tau_{co}}{\to}x$ and $x\ne y\in\bigcap_{m}\overline{\conv(x_{n})_{n\ge m}}$
then there is a quasi-monotone functional $\ell$ such that $\ell(y)>\ell(x)$.
Since $\ell$ is weakly continuous we have $\ell(x_{n})\to x$ and
thus by quasi-convexity of $\ell$ also $\ell(y)>\ell(x')$ for all
$x'\in\overline{\conv(x_{n})_{n\ge m}}$ with $m\in\mathbb{N}$ sufficiently
large $m\in\mathbb{N}$ However, this contradicts $y\in\bigcap_{m}\overline{\conv(x_{n})_{n\ge m}}$
and also shows that $\tau_{co}$ is Hausdorff.\end{proof}
\begin{thm}
[Nearly uniform convexity]\label{thm:NUC} Let $(X,d)$ be nearly
uniformly convex and countable reflexive. Then for any $\epsilon$-separated
sequence $(x_{n})$ in $B_{R}(y)$ there is a weak limit point of
$(x_{n})$ contained in the ball $B_{(1-\rho)R}(y)$.\end{thm}
\begin{proof}
If $(x_{n})$ is $\epsilon$-separated with $d(x_{n},y)\le R$ and
assume w.l.o.g. that $(x_{n})$ is chosen such that
\[
\Lim(x_{n})=\bigcap_{m\in\mathbb{N}}C_{m}
\]
where $C_{m}=\overline{\conv(x_{n})_{n\ge m}}$. We know by nearly
uniform convexity there is a $\rho>0$ such that 
\[
\tilde{C}_{m}=B_{(1-\rho)R}(y)\cap C_{m}\ne\varnothing.
\]
Since $\tilde{C}_{m}$ is non-decreasing closed convex and non-empty,
we see by reflexivity that $\cap_{m}\tilde{C}_{m}\ne\varnothing$
and hence $B_{(1-\rho)R}(y)\cap\Lim(x_{n})\ne\varnothing$. \end{proof}
\begin{thm}
[Kadec-Klee property]Let $(X,d)$ be strictly $\infty$-convex, nearly
uniformly convex and countable reflexive. Suppose some fixed $y\in X$
and for each weak limit point $x$ of $(x_{n})$ one has $d(x_{n},y)\to d(x,y)$
then $(x_{n})$ has exactly one limit point and $(x_{n})$ converges
strongly, i.e. norm plus weak convergence implies strong convergence.\end{thm}
\begin{cor}
If, in addition, $\tau_{co}$ is Hausdorff then $x_{n}\overset{\tau_{co}}{\to}x$
and $d(x_{n},y)\to d(x,y)$ implies $x_{n}\to x$.\end{cor}
\begin{proof}
[Proof of the Theorem]Since strong convergence implies weak and norm
convergence, we only need to show the converse. For this let $x_{n}$
be some weakly convergent sequence. Note that $d(x,y)=\lim d(x_{n},y)=const$
for all limit points $x$ of $(x_{n})_{n}$. Since $x\mapsto d(\cdot,y)$
is strictly quasi-convex and the set of limit points is convex, there
can be at most one limit point, i.e. $x_{n}\overset{\tau_{co}}{\to}x$
for a unique $x\in X$. If $d(x,y)=0$ then $x=y$ and $x_{n}\to x$
strongly. 

Now assume $d(x,y)=R>0$. If $(x_{n})$ is not Cauchy then there is
a subsequence $(x_{n'})$ still weakly converging to $x$ which is
$\epsilon$-separated for some $\epsilon>0$. By the Theorem \ref{thm:NUC}
there is a limit point $x^{*}$ of $(x_{n'})$ such that $d(x^{*},y)<R$.
But this contradicts the fact that $x^{*}=x$ and $d(x,y)=R$. Therefore,
$x=y$.
\end{proof}

\subsection*{A ``topology'' via asymptotic centers}

A more popular notion of convergence is the weak sequential convergence.
Note, however, it is an open problem whether this ``topology'' is
actually generated by a topology, see \cite[Question 3.1.8.]{Bacak2014a}.
Given a sequence $(x_{n})$ in $X$ define the following function
\[
\omega(x,(x_{n}))=\limsup_{n\to\infty}d(x,x_{n}).
\]

\begin{lem}
Assume $(X,d)$ is uniformly $\infty$-convex. Then function $\omega(\cdot,(x_{n}))$
has a unique minimizer.\end{lem}
\begin{proof}
It is not difficult to see that the sublevels of $\omega(\cdot,(x_{n}))$
are closed bounded and convex. This reflexivity implies existence
of minimizers. Assume $x,x'$ are minimizers and $x_{\frac{1}{2}}$
their midpoint. If $\omega(x,(x_{n}))=0$ then obviously $x_{n}\to x=x'$.
So assume $\omega(x,(x_{n})_{n})=c>0$. Then we can choose a subsequence
$(x_{n'})_{n'}$ such that $\lim_{n'\to\infty}d(x,x_{n'})$ and $\lim_{n'\to\infty}d(x',x_{n'})$
exists and are equal. If $x\ne x'$ then there is an $\epsilon>0$
such that $d(x,x')\ge2\epsilon c$. This yields 
\begin{eqnarray*}
\limsup_{n'\to\infty}d(x_{\frac{1}{2}},x_{n'}) & \le & (1-\rho(\epsilon))\lim_{n'\to\infty}\max\{d(x,x_{n'}),d(x',x_{n'})\}\\
 & < & c.
\end{eqnarray*}
But this contradicts $x$ and $x'$ be minimizers. Hence $x=x'$.
\end{proof}
The minimizer of $\omega(\cdot,(x_{n}))$ is called the asymptotic
center. With the help of this we can define the weak sequential convergence
as follows.
\begin{defn}
[Weak sequential convergence] We say that a sequence $(x_{n})_{n\in\mathbb{N}}$
converges weakly sequentially to a point $x$ if $x$ is the asymptotic
center for each subsequence of $(x_{n})$. We denote this by $x_{n}\overset{w}{\to}x$.

For $CAT(0)$-spaces it is easy to see that $x_{n}\overset{w}{\to}x$
implies $x_{n}\overset{\tau_{co}}{\to}x$, i.e. the weak topology
is weaker than the weak sequential convergence (see \cite[Lemma 3.2.1]{Bacak2014a}).
Later we will show that the weak sequential limits can be strongly
approximated by barycenters, which can be seen as a generalization
of the Banach-Saks property (see below). If, in addition, the barycenter
of finitely many points is in the convex hull of those points, one
immediate gets that $x_{n}\overset{w}{\to}x$ implies $x_{n}\overset{\tau_{co}}{\to}x$.\end{defn}
\begin{prop}
Each bounded sequence $(x_{n})$ has a subsequence $(x_{n'})$ such
that $x_{n'}\overset{w}{\to}x$.\end{prop}
\begin{proof}
The proof can be found in \cite[Proposition 3.2.1]{Bacak2014a}. Since
it is rather technical we leave it out.
\end{proof}
A different characterization of this convergence can be given as follows
(see \cite[Proposition 3.2.2]{Bacak2014a}).
\begin{prop}
Assume $(X,d)$ is uniformly $\infty$-convex. Let $(x_{n})$ be a
bounded sequence and $x\in X$. The the following are equivalent:
\begin{enumerate}
\item The sequence $(x_{n})$ converges weakly sequentially to $x$
\item For every geodesic $\gamma:[0,1]\to X$ with $x\in\gamma_{[0,1]}$,
we have $P_{\gamma}x_{n}\to x$ as $n\to\infty$.
\item For every $y\in X$, we have $P_{[x,y]}x_{n}\to x$ as $n\to\infty$.
\end{enumerate}
\end{prop}
\begin{proof}
(i)$\Longrightarrow$(ii): Let $\gamma$ be some geodesic containing
$x$. If 
\[
\lim d(P_{\gamma}x_{n},x)\ge0
\]
then there is a subsequence $(x_{n'})$ such that 
\[
P_{\gamma}y_{n}\to y\in\gamma_{[0,1]}\backslash\{x\}.
\]
But then $d(P_{\gamma}x_{n'},x_{n'})<d(x,x_{n'})$ which implies 
\[
\limsup_{n\to\infty}d(y,x_{n'})=\limsup_{n\to\infty}d(P_{\gamma}x_{n'},x_{n'})\le\limsup d(x,x_{n'})
\]
and contradicts uniqueness of the asymptotic center of $(x_{n'})$. 

(ii)$\Longrightarrow$(iii): Trivial

(iii)$\Longrightarrow$(i): Assume $(x_{n})$ does not converge weakly
sequentially to $x$. Then for some subsequence $x_{n'}\overset{w}{\to}y\in X\backslash\{x\}$.
Then by the part above $P_{[x,y]}x_{n'}\to y$. But this contradicts
the assumption $P_{[x,y]}x_{n'}\to x$. Hence $x_{n}\overset{w}{\to}x$.\end{proof}
\begin{cor}
[Opial property] \label{cor:Opial}Assume $(X,d)$ is uniformly $\infty$-convex
and $(x_{n})$ some bounded sequence with $x_{n}\overset{w}{\to}x$.
Then 
\[
\liminf d(x,x_{n})<\liminf d(y,x_{n})
\]
for all $y\in X\backslash\{x\}$. 
\end{cor}

\section*{Barycenters in convex metric spaces}

\subsection*{Wasserstein space}

For $p\in[1,\infty)$ the $p$-Wasserstein space of a metric space
$(X,d)$ is defined as the set $\mathcal{P}_{p}(X)$ of all probability
measures $\mu\in\mathcal{P}(X)$ such that 
\[
\int d^{p}(x,x_{0})d\mu(x)
\]
for some fixed $x_{0}\in X$. Note that by triangle inequality this
definition is independent of $x_{0}$. We equip this set with the
following metric 
\[
w_{p}(\mu,\nu)=\left(\inf_{\pi\in\Pi(\mu,\nu)}\int d^{p}(x,y)d\pi(x,y)\right)^{\frac{1}{p}}
\]
where $\Pi(\mu,\nu)$ is the set of all coupling measures $\pi\in\mathcal{P}(X\times X)$
such that $\pi(A\times X)=\mu(A)$ and $\pi(X\times B)=\nu(B)$. It
is well-known \cite{Villani2009} that $(\mathcal{P}_{p}(X),w_{p})$
is a complete metric space if $(X,d)$ is complete and that it is
a geodesic space if $(X,d)$ is geodesic. Furthermore, by Hölder inequality
one easily sees that $w_{p}\le w_{p'}$ whenever $p\le p'$ so that
the limit 
\[
w_{\infty}(\mu,\nu)=\lim w_{p}(\mu,\nu)
\]
is well-defined and defines a metric on the space $\mathcal{P}_{\infty}(X)$
of probability measures with bounded support. An equivalent description
of $w_{\infty}$ can be given as follows (see \cite{Champion2008}):
For a measure $\pi\in\Pi(\mu,\nu)$ let $C(\pi)$ be the $\pi$-essiential
support of $d(\cdot,\cdot)$, i.e.
\[
C(\pi)={\pi-\esup}_{(x,y)\in X\times X}d(x,y).
\]
Then 
\[
w_{\infty}(\mu,\nu)=\inf_{\pi\in\Pi(\mu,\nu)}C(\pi).
\]

For a fixed point $y\in X$ the distance of $\mu$ to the delta measure
$\delta_{y}$ has the following form
\[
w_{p}^{p}(\mu,\delta_{y})=\int d^{p}(x,y)d\mu(x)
\]
and 
\[
w_{\infty}(\mu,\delta_{y})=\sup_{x\in\supp\mu}d(x,y)
\]
where $\supp\mu$ is the support of $\mu$.

\subsection*{Existence and uniqueness of barycenters}
\begin{lem}
\label{lem:measure-conv}Assume $(X,d)$ is $p$-convex then $y\mapsto\int d^{p}(x,y)d\mu(x)$
is convex for $p\in[1,\infty)$ whenever $\mu\in\mathcal{P}_{p}(X)$.
In case $p>1$ strict $p$-convexity even implies strict convexity.
Furthermore, if $\mu$ is not supported on a single geodesic then
$y\mapsto\int d(x,y)d\mu(x)$ is strictly convex if $(X,d)$ is strictly
$1$-convex.\end{lem}
\begin{rem*}
(1) It is easy to see that for a measure supported on a geodesic the
functional $y\mapsto\int d(x,y)d\mu(x)$ cannot be strictly convex
on that geodesic. 

(2) The same holds for the functional $F_{w}(y):=\int d^{p}(x,y)-d^{p}(x,w)d\mu(x)$
as defined in \cite{Kuwae2013}\end{rem*}
\begin{proof}
Let $y_{0},y_{1}\in X$ be two point in $X$ and $y_{t}$ be any geodesic
connecting $y_{0}$ and $y_{1}$. Then by $p$-convexity 
\[
d^{p}(x,y_{t})\le(1-t)d^{p}(x,y_{0})+td^{p}(x,y_{1})
\]
which implies convexity of the functional and similarly strict convexity
if $p>1$. 

If $\mu$ is not supported on a single geodesic then there is a subset
of positive $\mu$-measure disjoint from $\{y_{t}|t\in[0,1]\}$ such
that $d(x,y_{t})<(1-t)d(x,y_{0})+td(x,y_{1})$. In particular, $y\mapsto\int d(x,y)d\mu(x)$
is strictly convex.\end{proof}
\begin{lem}
Assume $(X,d)$ is uniformly $\infty$-convex with modulus $\rho$.
Let $\mu\in\mathcal{P}_{\infty}(X)$ then the function $F:y\mapsto w_{\infty}(\mu,\delta_{y})$
is uniformly quasi-convex, i.e. whenever $d(y_{0},y_{1})>\epsilon\max\{F(y_{0}),F(y_{1})\}$
for some $\epsilon>0$ then 
\[
F(y_{\frac{1}{2}})\le(1-\rho(\epsilon))\max\{F(y_{0}),F(y_{1})\}.
\]
\end{lem}
\begin{rem*}
In contrast to the cases $1<p<\infty$ strict $\infty$-convexity
is not enough.\end{rem*}
\begin{proof}
Note that $F$ has the following equivalent form
\[
F(y)=\sup_{x\in\supp\mu}d(x,y).
\]
Take any $y_{0},y_{1}\in X$ with $d(y_{0},y_{1})>\epsilon\max\{F(y_{0}),F(y_{1})\}$.
Let $x_{n}$ be a sequence such that $F(y_{\frac{1}{2}})=\lim_{n\to\infty}d(x_{n},y_{\frac{1}{2}})$.
By uniform $\infty$-convexity we have 
\begin{eqnarray*}
\lim_{n\to\infty}d(x_{n},y_{\frac{1}{2}}) & \le & (1-\rho(\epsilon))\max\{d(x_{n},y_{0}),d(x_{n},y_{1})\}\\
 & \le & (1-\rho(\epsilon))\max\{F(y_{0}),F(y_{1})\}.
\end{eqnarray*}

\end{proof}
The following was defined in \cite{Kuwae2013}.
\begin{defn}
[$p$-barycenter] For $p\in[1,\infty]$ the $p$-barycenter of a measure
$\mu\in\mathcal{P}_{p}(X)$ is defined as the point $y\in X$ such
that $w_{p}(\mu,\delta_{y})$ is minimal. If $p<\infty$ and $\mu$
has only $(p-1)$-moments, i.e. $\int d^{p-1}(x,y)d\mu(x)<\infty$,
then the $p$-barycenter can be defined as the minimizer of the functional
$F_{w}(y)$ above. If the $p$-barycenter is unique we denote it by
$b_{p}(\mu)$.\end{defn}
\begin{rem*}
(1) This functional $F_{w}(y)$ is well-defined since 
\[
|F_{w}(y)|\le pd(y,w)\int(d(x,y)+d(x,w))^{p-1}d\mu(x).
\]
 Furthermore, $F_{w}(y)-F_{w'}(y)$ is constant and thus the minimizer(s)
are independent of $w\in X$.

(2) The $\infty$-barycenters are also called circumcenter. In case
$\mu$ consists of three points it was recently used in \cite{Bacak2014}
to define a new curvature condition. From the section above, the $\infty$-barycenter
only depends on the support of the measure $\mu$. Hence the $\infty$-barycenter
of any bounded set $A$ can be defined as 
\[
b_{\infty}(A)=\argmin_{y\in X}\sup_{x\in A}d(x,y).
\]
The proofs below work without any change.\end{rem*}
\begin{thm}
On any $p$-convex, reflexive metric space $(X,d)$ every measure
$p$-moment has $p$-barycenter.\end{thm}
\begin{proof}
Define
\[
A_{r}^{p}=\{y\in X\,|\, w_{p}(\mu,\delta_{y})\le r\},
\]
which is a closed convex subset of $X$ which is non-empty for $r>m_{\mu}^{p}=\inf_{y\in X}w_{p}(\mu,\delta_{y})$.
If it is bounded then by reflexivity 
\[
A_{m_{\mu}}=\bigcap_{r>m_{\mu}}A_{r}\ne\varnothing.
\]
In this case minimality implies $w_{\infty}(\mu,\delta_{y})=m_{\mu}$
for all $y\in A_{m_{\mu}}$.

In case $p=\infty$ note that $y\mapsto w_{\infty}(\mu,\delta_{y})=\sup_{x\in\supp\mu}d(x,y)$
is finite iff $\mu$ has bounded support in which case $A_{r}$ is
bounded as well.

The cases $p\in(1,\infty)$ where proven in \cite[Proposition 3.1]{Kuwae2013},
the assumption on properness can be dropped using reflexivity. For
convenience we include the short proof: If $\mu\in\mathcal{P}_{p}(X)$
then $w_{p}(\mu,\delta_{y_{0}})\le R$. Now take any $y\in X$ and
assume $w_{p}(\mu,\delta_{y})\le r$. Since $(X,d)$ is isometrically
embedded into $(\mathcal{P}_{p}(X),w_{p})$ by the map $y\mapsto\delta_{y}$
we have
\begin{eqnarray*}
d(y_{0},y)=w_{p}(\delta_{y_{0}},\delta_{y}) & \le & w_{p}(\mu,\delta_{y_{0}})+w_{p}(\mu,\delta_{y})\\
 & \le & R+r,
\end{eqnarray*}
i.e. $y\in B_{R+r}(y_{0})$ which implies $A_{r}$ is bounded. Using
a similar argument one can also show that $A_{r}$ is bounded if $\mu$
is only in $\mathcal{P}_{p-1}(X)$.\end{proof}
\begin{cor}
Let $p\in[1,\infty]$ and $(X,d)$ be a strictly $p$-convex if $p\in[1,\infty)$
and uniformly $\infty$-convex if $p=\infty$. Then $p$-barycenters
are unique for $p>1$. In case $p=1$, all measure admitting $1$-barycenters
which are not supported on a single geodesic have a unique $1$-barycenter.
\end{cor}
The $p$-product of finitely many metric spaces $\{(X_{i},d_{i})\}_{i=1}^{n}$
is defined the metric space $(X,d)$ with $X=\times_{i=1}^{n}X_{i}$
and 
\[
d(x,y)=\left(\sum_{i=1}^{n}d_{i}^{p}(x_{i},y_{i})\right)^{\frac{1}{p}}.
\]
A minor extension of \cite[Theorem 1]{Foertsch2004} shows that for
$p\in(1,\infty)$ the space $(X,d)$ is strictly $p$-convex if all
$(X_{i},d_{i})$ are if $1<p<\infty$ and projections onto the factors
of geodesic in $(X,d)$ are geodesics in $(X_{i},d_{i})$.
\begin{thm}
Let $\{(X_{i},d_{i})\}_{i=1}^{n}$ be finitely many strictly $p$-convex
reflexive metric spaces and $(X,d)$ be the $p$-product of $\{(X_{i},d_{i})\}_{i=1}^{n}$.
If $\mu\in\mathcal{P}_{p}(X)$ then $b_{p}(\mu)=(b_{p}(\mu_{i}))$
where $\mu_{i}$ are the marginals of $\mu$. \end{thm}
\begin{proof}
If is not difficult to see that 
\[
w_{p}^{p}(\mu,\delta_{y})=\sum_{i=1}^{n}w_{p}^{p}(\mu_{i},\delta_{y_{i}}).
\]
Thus by existence for the factors we know 
\[
\inf_{y\in X}w_{p}(\mu,\delta_{y})\le\sum_{i=1}^{n}w_{p}^{p}(\mu_{i},\delta_{b_{p}(\mu_{i})})=w_{p}^{p}(\mu,(b_{p}(\mu_{i}))).
\]
Conversely, suppose there is a $y$ such that $w_{p}(\mu,\delta_{y})\le w_{p}(\mu,b_{p}(\mu))$.
Since it holds $w_{p}(\mu_{i},\delta_{y_{i}})\ge w_{p}(\mu_{i},b_{p}(\mu_{i}))$
we see that $y$ is a minimizer of $y\mapsto w_{p}(\mu,\delta_{y})$.
Since $(X,d)$ is strictly $p$-convex, $y=b_{p}(\mu)$.
\end{proof}

\subsection*{Jensen's inequality}

The classical Jensen's inequality states that on a Hilbert space $H$
for any measure $\mu\in\mathcal{P}_{1}(H)$ and any convex lower semicontinuous
function $\varphi\in L^{1}(H,\mu)$ it holds
\[
\varphi\left(\int xd\mu(x)\right)\le\int\varphi(x)d\mu(x).
\]
With the help of barycenters Jensen's inequality can be stated as
follows.
\begin{defn}
[Jensen's inequality] A metric space $(X,d)$ is said to admit the
$p$-Jensen's inequality if for all measure $\mu$ admitting a (unique)
barycenter $b_{p}(\mu)$ and for every lower semicontinuous function
$\varphi\in L^{p-1}(X,\mu)$ it holds
\[
\varphi(b_{p}(\mu))\le b_{p}(\varphi_{*}\mu)
\]
where $\varphi_{*}\mu\in\mathcal{P}(\mathbb{R})$ is the push-forward
of $\mu$ via $\varphi$.

For $p=2$ this boils down to 
\[
b_{2}(\varphi_{*}\mu)=\int\varphi d\mu.
\]

\end{defn}
Using the existence proofs above one can adapt Kuwae's proof of \cite[Theorem 4.1]{Kuwae2013}
to show that Jensen's inequality holds spaces satisfying the condition
$(\mathbf{B})$, i.e. for any two geodesics $\gamma,\eta$ with $\{p_{0}\}=\gamma\cap\eta$
and $\pi_{\gamma}(y)=p_{0}$ for $y\in\eta\backslash\{p_{0}\}$ it
holds $\pi_{\eta}(x)=p_{0}$ for all $x\in\eta$. Kuwae states this
as $\eta\bot_{p_{0}}\gamma$ implies $\gamma\bot_{p_{0}}\eta$. 

Using Busemann's non-positive curvature condition one can then show
that if each measure in $\mathcal{P}_{1}(X)$ admits unique barycenters
and Jensen's inequality holds on the product space then a Wasserstein
contraction holds, i.e. 
\[
d(b_{2}(\mu),b_{2}(\nu))\le w_{1}(\mu,\nu).
\]

If instead the $p$-Busemann holds one still gets the following:
\begin{prop}
\label{prop:Jensen-p-Busemann}Let $(X,d)$ be $p$-Busemann for some
$p\in[1,\infty)$. If the $2$-Jensen's inequality for holds on the
$2$-product $X\times X$ then 
\[
d(b_{2}(\mu),b_{2}(\nu))\le w_{p}(\mu,\nu).
\]
\end{prop}
\begin{proof}
Since $(x,y)\mapsto d(x,y)^{p}$ is convex on $X\times X$ we have
by Jensen's inequality
\[
d(b_{2}(\mu),b_{2}(\nu))^{p}\le\int d^{p}(x,y)d\pi(x,y)
\]
for any $\pi\in\Pi(\mu,\nu)$. Hence $d(b_{2}(\mu),b_{2}(\nu))\le w_{p}(\mu,\nu)$.
\end{proof}

\section*{Banach-Saks}

The classical Banach-Saks property for Banach spaces is stated as
follows: Any bounded sequence has a subsequence $(x_{m_{n}})$ such
that sequence of Ces\`aro means
\[
\frac{1}{N}\sum_{n=1}^{N}x_{m_{n}}
\]
converges strongly. In a general metric space there is no addition
of two elements defined. Furthermore, convex combinations do not commute
(are not associative), i.e. if $(1-\lambda)x\oplus\lambda y$ denotes
the point $x_{\lambda}$ on the geodesic connecting $x$ and $y$
then in general 
\[
\frac{2}{3}\left(\frac{1}{2}x\oplus\frac{1}{2}y\right)\oplus\frac{1}{3}z\ne\frac{1}{3}x\oplus\frac{2}{3}\left(\frac{1}{2}y\oplus\frac{1}{2}z\right),
\]
so that $\frac{1}{N}\bigoplus_{n=1}^{N}x_{n}$ does not make sense. 

For Hilbert spaces the point $\frac{1}{N}\bigoplus_{n=1}^{N}x_{n}$
agrees with the $2$-barycenter of the measure 
\[
\mu_{N}=\frac{1}{N}\sum_{n=1}^{N}\delta_{x_{n}}.
\]
Since this is well-defined on general metric spaces the Banach-Saks
property can be formulated as follows.
\begin{defn}
[$p$-Banach-Saks] Let $p\in[1,\infty]$ and suppose for any sequence
$(x_{n})$ in a metric space $(X,d)$ the measures $\mu_{n}$ admit
a unique $p$-barycenter. Then $(X,d)$ is said to satisfy the $p$-Banach-Saks
property if every sequence $(x_{n})$ there is a subsequence $(x_{m_{n}})$
such that the sequence of $p$-barycenters of the measures $\tilde{\mu}_{N}=\frac{1}{N}\sum_{n=1}^{N}\delta_{x_{m_{n}}}$
converges strongly.
\end{defn}
Since Hilbert spaces satisfy the (traditional) Banach-Saks property
they also satisfy the $p$-Banach-Saks property. Yokota managed in
\cite[Theorem C]{Yokota} (see also \cite[Theorem 3.1.5]{Bacak2014a})
to show that any $CAT(1)$-domain with small radius, in particular
any $CAT(0)$-space, satisfy the $2$-Banach-Saks property and if
$x_{n}\overset{w}{\to}x$ then $b_{p}(\tilde{\mu}_{N})\to x$ where
$\tilde{\mu}_{N}$ is defined above. We will adjust his proof to show
that for $p\in(1,\infty)$ any uniformly $p$-convex space satisfies
the $p$-Banach-Saks property and the limit of the chosen subsequence
agrees with the weak sequential limit. Since a $CAT(1)$-domain with
small radius is uniformly $p$-convex, our result generalizes Yokota's
when restricted to that convex subset.

We leave the proof of the following statement to the reader.
\begin{lem}
If $(X,d)$ is uniformly $p$-convex and $x_{n}\to x$ then $b_{p}(\mu_{N})\to x$
where $\mu_{N}=\frac{1}{N}\sum_{n=1}^{N}\delta_{x_{n}}$.
\end{lem}

\begin{lem}
\label{lem:unif-conv}Assume $(X,d)$ be a metric space admitting
midpoints and let $f:X\to\mathbb{R}$ be a uniformly convex function
with modulus $\omega$. If $f$ attains its minimum at $x_{m}\in X$
then 
\[
f(x)\ge f(x_{\min})+\frac{1}{2}\omega(d(x,x_{\min})).
\]
\end{lem}
\begin{proof}
Let $x\in X$ be arbitrary and $m$ be the midpoint of $x$ and $x_{m}$.
Then 
\[
f(x_{\min})\le f(m)\le\frac{1}{2}f(x_{\min})+\frac{1}{2}f(x)-\frac{1}{4}\omega(d(x,x_{\min})).
\]
\end{proof}
\begin{thm}
\label{thm:Banach-Saks}Let $(X,d)$ be a uniformly $p$-convex metric
space. If $x_{n}\overset{w}{\to}x$ then there is a subsequence $(x_{m_{n}})$
such that $b_{p}(\tilde{\mu}_{N})\to x$ where $\tilde{\mu}_{N}=\frac{1}{N}\sum_{n=1}^{N}\delta_{x_{m_{n}}}$.
In particular, $(X,d)$ satisfies the $p$-Banach-Saks property.\end{thm}
\begin{proof}
By the previous lemma we can assume that $(x_{n})$ is $2\epsilon$-separated
for some $\epsilon>0$, since otherwise there is a strongly convergent
subsequence fulfilling the statement of the theorem. Furthermore,
assume w.l.o.g. that $d(x_{n},x)\to r$. Then 
\[
\lim_{N\to\infty}\frac{1}{N}\sum d(x,x_{n})^{p}\to r^{p}.
\]

For any measure $\mu\in\mathcal{P}_{p}(X)$ define 
\[
\Var_{\mu,p}(y)=\int d(x,y)^{p}d\mu(x)
\]
and 
\[
V(\mu)=\inf_{y\in X}\Var_{\mu,p}(y)
\]
Furthermore, for any finite subset $I\subset\mathbb{N}$ define 
\[
\mu_{I}=\frac{1}{|I|}\sum_{i\in I}\delta_{x_{i}}.
\]
Let $R>0$ be such that $d(x_{n},x)\le R$. Note that $b_{p}(\mu_{I})\in B_{2R}(x)$
for any finite $I\subset\mathbb{N}$. Furthermore, if $(y_{t})_{t\in[0,1]}$
is a geodesic in $B_{2R}(x)$ with $d(y_{0},y_{1})\ge3\delta R$ then
$d(y_{0},y_{1})\ge\delta\mathcal{M}^{p}(d(x_{n},y_{0}),d(x_{n},y_{1}))$
and by uniform $p$-convexity 
\[
d(x_{n},y_{\frac{1}{2}})^{p}\le(1-\tilde{\rho}_{p}(\delta))\left(\frac{1}{2}d(x_{n},y_{0})+\frac{1}{2}d(x_{n},y_{1})\right).
\]
Thus there is a monotone function $\tilde{\omega}:(0,\infty)\to(0,\infty)$
such that
\[
d(x_{n},y_{\frac{1}{2}})^{p}\le(1-\tilde{\omega}(d(y_{0},y_{1}))\left(\frac{1}{2}d(x_{n},y_{0})^{p}+\frac{1}{2}d(x_{n},y_{1})^{p}\right).
\]
This implies 
\[
\Var_{\mu_{I},p}(y_{\frac{1}{2}})\le(1-\tilde{\omega}(d(y_{0},y_{1}))\left(\frac{1}{2}\Var_{\mu_{I},p}(y_{0})+\frac{1}{2}\Var_{\mu_{I},p}(y_{1})\right)
\]

By $2\epsilon$-separation of $(x_{n})$ for any finite $I\subset\mathbb{N}$
there is at most one $i\in I$ such that $d(x_{i},y)\le\epsilon$.
Hence if $|I|\ge2$ 
\[
V(\mu_{I})\ge\frac{1}{2}\epsilon^{p}.
\]
Combining this with the above inequality we see that for $\omega(r)=2\epsilon^{p}\tilde{\omega}(r)$
\[
\Var_{\mu_{I},p}(y_{\frac{1}{2}})\le\frac{1}{2}\Var_{\mu_{I},p}(y_{0})+\frac{1}{2}\Var_{\mu_{I},p}(y_{1})-\frac{1}{4}\omega(d(y_{0},y_{1})).
\]
This implies 
\[
\Var_{\mu_{I},p}(y_{t})\le(1-t)\Var_{\mu_{I},p}(y_{0})+t\Var_{\mu_{I},p}(y_{1})-t(1-t)\omega(d(y_{0},y_{1})).
\]
i.e. the functions $\Var_{\mu_{I},p}:B_{2R}(X)\to\mathbb{R}$ are
uniformly convex with modulus $\omega$. 

The next steps follow directly from the proofs of \cite[Theorem C]{Yokota}
and \cite[Theorem 3.1.5]{Bacak2014a} we include the whole proof for
convenience of the reader.

Step 1: set $I_{k}^{N}=\{(k-1)2^{N},\ldots,k2^{N}\}\subset\mathbb{N}$
for any $k,N\in\mathbb{N}$. We claim that if 
\[
\sup_{N\in\mathbb{N}}\liminf_{k\to\infty}V(\mu_{I_{k}^{N}})=r^{p}
\]
then $b_{n}=b_{p}(\mu_{I_{n}^{0}})$ converges strongly to $x$. To
see this note that for any $\epsilon>0$ there is an $N\in\mathbb{N}$
such that 
\[
\liminf_{k\to\infty}V(I_{k}^{N})\ge r^{p}-\epsilon.
\]
Then 
\begin{eqnarray*}
\liminf_{n\to\infty}\Var_{\mu_{I_{n}^{0}},p}(b_{n}) & \ge & \liminf_{k\to\infty}V(\mu_{I_{k}^{N}})\\
 & \ge & r^{p}-\epsilon.
\end{eqnarray*}
Since $\epsilon>0$ is arbitrary we see that $\liminf_{n\to\infty}\Var_{\mu_{I_{n}^{0}},p}(b_{n})\ge r^{p}$.

By uniform convexity of $\Var_{\mu_{I_{n}^{0}},p}$ and Lemma \ref{lem:unif-conv}
we also have 
\[
\Var_{\mu_{I_{n}^{0}},p}(x)\ge\Var_{\mu_{I_{n}^{0}},p}(b_{n})+\frac{1}{2}\omega(d(x,b_{n})).
\]
Since the left hand side converges to $r^{p}$ and $\limsup_{n\to\infty}\Var_{\mu_{I_{n}^{0}},p}(b_{n})\ge r^{p}$.
This implies that $\limsup\omega(d(x,b_{n}))\to0$, i.e. $d(x,b_{n})\to0$.

Step 2: We will select a subsequence of $(x_{n})$ such the assumption
of the claim in Step 1 are satisfied. Set $J_{k}^{0}=\{k\}$ for $k\in\mathbb{N}$.
We construct a sequence of set $J_{k}^{N}$ for $N\in\mathbb{N}$
of cardinality $2^{N}$ such that $J_{k}^{N}=J_{l}^{N-1}\cup J_{m}^{m-1}$
for some $m,l\in N$. Furthermore,$\max J_{k}^{N}<J_{k+1}^{N}$ and
\[
\lim_{k\to\infty}V(\mu_{J_{k}^{N}})=V_{N}:=\limsup_{l,m\to\infty}V(\mu_{J_{l}^{N-1}\cup J_{m}^{N-1}}).
\]
It is not difficult to see that $V_{N}\le V_{N+1}\le r^{p}$ for every
$N\in\mathbb{N}$. We will show that $V_{N}\to r^{p}$ as $N\to\infty$.
It is not difficult to see that this follows from the claim below.
\begin{claim*}
For every $\epsilon'>0$ there exists a $\delta>0$ such that whenever
$V_{N}<(r-\epsilon)^{p}$ then $V_{N+1}>V_{N}+\delta$.\end{claim*}
\begin{proof}
[Proof of claim] Fix $N\in\mathbb{N}$ and for $l\in\mathbb{N}$
let $b_{l}^{N}$ be the $p$-barycenter of $\mu_{J_{k}^{N}}$. By
assumption there is an $l\in\mathbb{N}$ such that $V(\mu_{J_{k}^{N}})<(r-\epsilon')^{p}$
for all $k\ge l$. By the Opial property, Corollary \ref{cor:Opial},
there is a large $m>l$ such that $d(b_{l}^{N},x_{i})>r$ for $i\in J_{m}^{N}$. 

This implies 
\[
\frac{1}{2^{N}}\sum_{i\in J_{m}^{N}}d(b_{l}^{N},x_{i})^{p}>r^{p}>(r-\epsilon)^{p}>V(\mu_{J_{m}^{N}})=\frac{1}{2^{N}}\sum_{i\in J_{m}^{N}}d(b_{m}^{N},x_{i})
\]
and hence 
\[
2\max\left\{ d(b_{m}^{N},b_{m\cup l}^{N}),d(b_{l}^{N},b_{m\cup l}^{N})\right\} \ge d(b_{m}^{N},b_{l}^{N})>\epsilon'
\]
where $b_{m\cup l}^{N}$ is the $p$-barycenter of $\mu_{J_{m}^{N}\cup J_{l}^{N}}$.
By uniform convexity of $\Var_{\mu_{I},p}$ and Lemma \ref{lem:unif-conv}
we get 
\begin{eqnarray*}
V(\mu_{J_{l}^{N}\cup J_{m}^{N}}) & = & \frac{1}{2^{N+1}}\left(\sum_{i\in J_{l}^{N}}d(b_{l\cup m}^{N},x_{i})^{p}+\sum_{i\in J_{l}^{N}}d(b_{l\cup m}^{N},x_{i})^{p}\right)\\
 & \ge & \frac{1}{2}\left[V(\mu_{J_{l}^{N}})+V(\mu_{J_{m,}^{N}})+\omega(d(b_{m}^{N},b_{m\cup l}^{N}))+\omega(d(b_{l}^{N},b_{m\cup l}^{N}))\right]\\
 & \ge & \frac{1}{2}\left[V(\mu_{J_{l}^{N}})+V(\mu_{J_{m,}^{N}})+2\omega(\epsilon')\right].
\end{eqnarray*}
 
\end{proof}
To finish the proof of the theorem, note that $\bigcap_{N}\bigcup_{k}J_{k}^{N}\subset\mathbb{N}$
is infinite. Denoting its elements in increasing order $(n_{1},n_{2},\ldots)$
we see that the sequence $(x_{n_{k}})$, after naming, satisfies the
assumption needed in Step 1.\end{proof}
\begin{cor}
Assume $(X,d)$ is uniformly $p$-convex and that for any sequence
$(x_{n})$ the $p$-barycenter of $\mu_{I}$ for any finite $I\subset\mathbb{N}$
is in the convex hull of the point $\{x_{i}\}_{i\in I}$. Then whenever
$x_{n}\overset{w}{\to}x$ implies $x\in\overline{\conv(x_{n})}$.
In particular, it holds $x_{n}\overset{\tau_{co}}{\to}x$, i.e. the
co-convex topology is weaker that weak sequential topology.\end{cor}
\begin{rem*}
The assumption of the corollary are satisfied on any $CAT(0)$-space,
see \cite[Lemma 2.3.3]{Bacak2014a} for the $2$-barycenter. More
generally Kuwae's condition $(\mathbf{A})$ is enough as well, see
\cite[Remark 3.7 (2)]{Kuwae2013}. In particular, it holds on all
spaces that satisfy Jensen's inequality.
\end{rem*}
\appendix

\section*{Generalized Convexities}

Let $L:(0,\infty)\to(0,\infty)$ be a strictly increasing convex function
such that $L(1)=1$ and $L(r)\to0$ as $r\to0$. Then $L$ has the
following form 
\[
L(r)=\int_{0}^{r}\ell(s)ds
\]
where $\ell$ is a positive monotone function. As an abbreviation
we also set $L_{\lambda}(r)=L(\frac{r}{\lambda})$ for $\lambda>0$.

Given $L$ we define the $L$-mean of two non-negative numbers $a,b\in[0,\infty)$
as follows
\[
\mathcal{M}^{L}(a,b)=L^{-1}\left(\frac{1}{2}L\left(a\right)+\frac{1}{2}L\left(b\right)\right)\left\{ t>0\,|\,\frac{1}{2}L\left(\frac{a}{t}\right)+\frac{1}{2}L\left(\frac{b}{t}\right)\le1\right\} .
\]
 
\begin{defn}
[$L$-convexity] A metric space admitting midpoints is said to be
$L$-convex if for any triple $x,y,z\in X$ it holds 
\[
d(m(x,y),z)\le\mathcal{M}^{L}(d(x,z),d(y,z)).
\]
If the inequality is strict whenever $x\ne y$ then the space is said
to be strictly $L$-convex.
\end{defn}
In a similar way one can use a more elaborate definition of mean:
For $L$ as above define the Orlicz mean 
\[
\tilde{\mathcal{M}}^{L}(a,b)=\inf\left\{ t>0\,|\,\frac{1}{2}L\left(\frac{a}{t}\right)+\frac{1}{2}L\left(\frac{b}{t}\right)\le1\right\} .
\]
 Now Orlicz $L$-convexity can be defined by using $\tilde{\mathcal{M}}^{L}$
instead of $\mathcal{M}^{L}$. It is not clear if this definition
is meaningful. The existence theorem for Orlicz-Wasserstein barycenters
below only uses $L$-convexity.

It is easy to see that for $p\in(1,\infty)$ (strict) $p$-convexity
is the same as (Orlicz) $L$-convexity for $L(r)=r^{p}$. However,
because $L$ needs to be strictly convex, the cases $1$-convexity
and $\infty$-convexity are not covered, but can be obtained as limits.
Strict convexity of $L$ also implies that the inequality above is
strict whenever $d(x,y)=|d(x,z)-d(y,z)|$, i.e. the condition $d(x,y)>|d(x,z)-d(y,z)|$
is not needed for strict $L$-convexity.
\begin{lem}
Suppose $\Phi$ is a convex function with $\Phi(1)=1$ and $\Phi(r)\to0$
as $r\to0$. Then any (strictly) (Orlicz) $L$-convex metric space
is (strictly) (Orlicz) $\Phi\circ L$-convex and (strictly) $\infty$-convex.
Also, any (strictly) $1$-convex space is strictly (Orlicz) $L$-convex.
\end{lem}
The proof of this lemma follows directly from convexity of $\Phi$.
Similarly one can define uniform convexity.
\begin{defn}
[uniform $p$-convexity]A strictly $L$-convex metric space is said
to be uniformly $L$-convex if for all $\epsilon>0$ there is a $\rho_{L}(\epsilon)\in(0,1)$
such that for all triples $x,y,z\in X$ satisfying $d(x,y)>\epsilon\mathcal{M}^{L}(d(x,z),d(y,z))$
it holds
\[
d(m(x,y),z)\le(1-\rho_{L}(\epsilon))\mathcal{M}^{L}(d(x,z),d(y,z)).
\]

\end{defn}
Using $L$ one can also define an Orlicz-Wasserstein space $(\mathcal{P}_{L}(X),w_{L})$,
see \cite{Sturm2011} and \cite[Appendix]{Kell2013a} for precise
definition and further properties. Since $L(1)=1$ the natural embedding
$x\to\delta_{x}$ is an isomorphism. For $\mu\in\mathcal{P}_{L}(X)$
and $y\in X$ the metric $w_{L}$ has the following form
\[
w_{L}(\mu,y)=\inf\{t>0\,|\,\int L\left(\frac{d(x,y)}{t}\right)d\mu(x)\le1\}.
\]
Note that by \cite{Sturm2011} the infimum is attained if $w_{L}(\mu,y)>0$.

Now the $L$-barycenter $b_{L}(\mu)$ of a measure $\mu\in\mathcal{P}_{L}(X)$
can be defined as 
\[
b_{L}(\mu)=\argmin_{y\in X}w_{L}(\mu,\delta_{y}).
\]

\begin{thm}
Assume $(X,d)$ is reflexive and strictly $L_{\lambda}$-convex for
any $\lambda>0$. Then any measure $\mu\in\mathcal{P}_{L}(X)$ admits
a unique barycenter.\end{thm}
\begin{rem*}
Since $L(r)=r^{p}$ is homogeneous, one sees that $(X,d)$ is strictly
$ $$L$-convex iff it is strictly $L_{\lambda}$-convex for some
$\lambda>0$.\end{rem*}
\begin{proof}
By our assumption we see that for any $\lambda>0$ 
\[
\int L\left(\frac{d(x,y_{\frac{1}{2}})}{\lambda}\right)d\mu(x)\le\frac{1}{2}\int L\left(\frac{d(x,y_{0})}{\lambda}\right)d\mu(x)+\frac{1}{2}\int L\left(\frac{d(x,y_{1})}{\lambda}\right)d\mu(x).
\]
with strict inequality whenever $y_{0}\ne y_{1}$. Hence, if $F_{\mu}(y_{0}),F_{\mu}(y_{1})\le\Lambda$
then 
\[
\int L\left(\frac{d(x,y_{\frac{1}{2}})}{\Lambda}\right)d\mu(x)\le1,
\]
i.e. $F_{\mu}(y_{\frac{1}{2}})\le\Lambda$. Furthermore, the strict
inequality is strict if $y_{0}\ne y_{1}$, i.e. $F_{\mu}$ is strictly
quasi-convex and can have at most one minimizer.

This now implies that the sublevels of $F_{\mu}$ are convex. Closedness
follows from continuity of $F_{\mu}$. In order to see that they are
also bounded, just note that $w_{L}$ is a metric, i.e. implies that
$|w_{L}(\mu,\delta_{y_{0}})-w_{L}(\delta_{y_{0}},\delta_{y})|\le w_{L}(\mu,\delta_{y})$.
Thus $F_{\mu}(y_{0})\le R$ for all $y\in X\backslash B_{2R}(y)$
it holds 
\[
F_{\mu}(y)=w_{L}(\mu,\delta_{y})>R.
\]

reflexivity implies now existence of $L$-barycenters. 
\end{proof}
In a similar way one can obtain a Banach-Saks theorem for spaces which
are uniformly $L_{\lambda}$-convex for each $\lambda>0$ such that
the moduli $(\rho_{L_{\lambda}})_{\lambda\in(0,\infty)}$ are equicomparable
for compact subsets of $(0,\infty)$. The proof then follows along
the line of Theorem \ref{thm:Banach-Saks}.

\bibliographystyle{amsalpha}
\bibliography{bib}

\newcommand{\etalchar}[1]{$^{#1}$}
\providecommand{\bysame}{\leavevmode\hbox to3em{\hrulefill}\thinspace}
\providecommand{\MR}{\relax\ifhmode\unskip\space\fi MR }
\providecommand{\MRhref}[2]{%
  \href{http://www.ams.org/mathscinet-getitem?mr=#1}{#2}
}
\providecommand{\href}[2]{#2}
\begin{thebibliography}{BHJ{\etalchar{+}}14}

\bibitem[Ba{\v{c}}14]{Bacak2014a}
M.~Ba{\v{c}}\'{a}k, \emph{{Convex analysis and optimization in Hadamard
  spaces}}, 2014.

\bibitem[BH99]{Bridson1999}
M.~R. Bridson and A.~H\"{a}fliger, \emph{{Metric Spaces of Non-Positive
  Curvature}}, Springer, 1999.

\bibitem[BHJ{\etalchar{+}}14]{Bacak2014}
M.~Ba\v{c}\'{a}k, B.~Hua, J.~Jost, M.~Kell, and A.~Schikorra, \emph{{A notion
  of nonpositive curvature for general metric spaces}}, arxiv:1404.0995 (2014).

\bibitem[BP79]{Busemann1979}
H.~Busemann and B.~B. Phadke, \emph{{Minkowskian geometry, convexity conditions
  and the parallel axiom}}, Journal of Geometry \textbf{12} (1979), no.~1,
  17--33.

\bibitem[CDJ08]{Champion2008}
Th. Champion, L.~{De Pascale}, and P.~Juutinen, \emph{{The $\infty$-Wasserstein
  Distance: Local Solutions and Existence of Optimal Transport Maps}}, SIAM
  Journal on Mathematical Analysis \textbf{40} (2008), no.~1, 1--20 (en).

\bibitem[Cla36]{Clarkson1936}
J.~A. Clarkson, \emph{{Uniformly convex spaces}}, Transactions of the American
  Mathematical Society \textbf{40} (1936), no.~3, 396--396.

\bibitem[EFL09]{Espinola2009}
R.~Esp\'{\i}nola and A.~Fern\'{a}ndez-Le\'{o}n, \emph{{CAT(k)-spaces, weak
  convergence and fixed points}}, Journal of Mathematical Analysis and
  Applications \textbf{353} (2009), no.~1, 410--427.

\bibitem[Foe04]{Foertsch2004}
T.~Foertsch, \emph{{Ball versus distance convexity of metric spaces}},
  Contributions to Algebra and Geometry (2004).

\bibitem[Huf80]{Huff1980}
R.~Huff, \emph{{Banach spaces which are nearly uniformly convex}}, Rocky
  Mountain J. Math (1980).

\bibitem[Kel13]{Kell2013a}
M.~Kell, \emph{{On Interpolation and Curvature via Wasserstein Geodesics}},
  arxiv:1311.5407 (2013).

\bibitem[KP08]{Kirk2008}
W.A. Kirk and B.~Panyanak, \emph{{A concept of convergence in geodesic
  spaces}}, Nonlinear Analysis: Theory, Methods \& Applications \textbf{68}
  (2008), no.~12, 3689--3696.

\bibitem[Kuw13]{Kuwae2013}
K.~Kuwae, \emph{{Jensen's inequality on convex spaces}}, Calculus of Variations
  and Partial Differential Equations \textbf{49} (2013), no.~3-4, 1359--1378.

\bibitem[Mon06]{Monod2006}
N.~Monod, \emph{{Superrigidity for irreducible lattices and geometric
  splitting}}, Journal of the American Mathematical Society \textbf{19} (2006),
  no.~4, 781--814.

\bibitem[NS11]{Noar2011}
A.~Noar and L.~Silberman, \emph{{Poincar\'{e} inequalities, embeddings, and
  wild groups}}, Compositio Mathematica \textbf{147} (2011), no.~05, 1546--1572
  (English).

\bibitem[Oht07]{Ohta2007}
S.~Ohta, \emph{{Convexities of metric spaces}}, Geometriae Dedicata
  \textbf{125} (2007), no.~1, 225--250.

\bibitem[Stu11]{Sturm2011}
K.-Th. Sturm, \emph{{Generalized Orlicz spaces and Wasserstein distances for
  convex-concave scale functions}}, Bulletin des Sciences Math\'{e}matiques
  \textbf{135} (2011), no.~6-7, 795--802.

\bibitem[Vil09]{Villani2009}
C.~Villani, \emph{{Optimal transport: old and new}}, Springer Verlag, 2009.

\bibitem[Yok13]{Yokota}
T.~Yokota, \emph{{Convex functions and barycenter on CAT(1)-spaces of small
  radii}}, Preprint available at http://www.kurims.kyoto-u.ac.jp/\~{}takumiy/
  (2013).

\end{thebibliography}

\end{document}